\documentclass[10pt]{amsart}
\usepackage[leqno]{amsmath}
\usepackage[numbers]{natbib}
\usepackage[mathcal]{euscript}
\usepackage{epsfig}
\usepackage{color}
\usepackage{multicol}
\usepackage{esint}

\usepackage[body={5.5in,8in}, top=1.5in,left=1.50in]{geometry}

\theoremstyle{plain}
\newtheorem{thm}{Theorem}
\newtheorem{cor}{Corollary}
\newtheorem{lem}[cor]{Lemma}
\newtheorem{prop}[cor]{Proposition}
\newtheorem{conj}[cor]{Conjecture}
\theoremstyle{definition}

\newtheorem{remark}[cor]{Remark}

\numberwithin{cor}{section}
\numberwithin{equation}{section}

\newcommand{\R}{\mathbb R}
\renewcommand{\d}{n}
\newcommand{\Rd}{\mathbb R^\d}
\newcommand{\ep}{\varepsilon}

\newcommand{\dist} {\mathrm{dist}}

\newcommand{\Sy}{{\mathcal S_\d}}
\newcommand{\Mat}{{\mathcal M_\d}}
\newcommand{\iden}{I_n}
\newcommand{\pucci}{{\mathcal{P}_{\lambda,\Lambda}^-}}
\newcommand{\Pucci}{{\mathcal{P}_{\lambda,\Lambda}^+}}

\DeclareMathOperator{\tr}{tr}

\begin{document}

\title[Partial regularity for fully nonlinear elliptic equations]{Partial regularity of solutions of fully nonlinear uniformly elliptic equations}
\author[S. N. Armstrong]{Scott N. Armstrong}
\address{Department of Mathematics\\ The University of Chicago\\ 5734 S. University Avenue
Chicago, Illinois 60637.}
\email{armstrong@math.uchicago.edu}
\author[L. Silvestre]{Luis Silvestre}
\address{Department of Mathematics\\ The University of Chicago\\ 5734 S. University Avenue
Chicago, Illinois 60637.}
\email{luis@math.uchicago.edu}
\author[C. K. Smart]{Charles K. Smart}
\address{Courant Institute of Mathematical Sciences \\ 251 Mercer Street \\New York, NY 10012.}
\email{csmart@cims.nyu.edu}
\date{\today}

\begin{abstract}
We prove that a viscosity solution of a uniformly elliptic, fully nonlinear equation is $C^{2,\alpha}$ on the compliment of a closed set of Hausdorff dimension at most $\ep$ less than the dimension. The equation is assumed to be $C^1$, and the constant $\ep > 0$ depends only on the dimension and the ellipticity constants. The argument combines the $W^{2,\ep}$ estimates of Lin with a result of Savin on the $C^{2,\alpha}$ regularity of viscosity solutions which are close to quadratic polynomials. 
\end{abstract}

\maketitle

\section{Introduction} \label{I}

In this paper, we prove a partial regularity result for viscosity solutions of the uniformly elliptic equation
\begin{equation} \label{ueeq}
F(D^2u) = 0.
\end{equation}
The operator $F$ is assumed to be uniformly elliptic and to have uniformly continuous first derivatives (these hypotheses are precisely stated in the next section). 

If $F$ is concave or convex, then solutions of \eqref{ueeq} in a domain $\Omega \subseteq\Rd$ are known to belong to $C^{2,\alpha}(\Omega)$ for some small $\alpha > 0$, according to the famous theorem of Evans~\cite{E} and Krylov~\cite{K} (see also \cite{CS} for a simple proof). Viscosity solutions of~\eqref{ueeq} have also been shown to be classical solutions for certain classes of nonconvex operators by Yuan~\cite{Y} as well as Cabr{\'e} and Caffarelli~\cite{CC1}. The latter result applies, for example, to an $F$ which is the minimum of a convex and a concave operator. However, $C^2$ estimates for solutions of \eqref{ueeq} are unavailable for general $F$, as attested by the recent counterexamples of Nadirashvili and Vl{\u{a}}du{\c{t}} \cite{NV1,NV2}. In fact, a counterexample to $C^{1,1}$ regularity was presented in \cite{NV2}, and therefore the best available regularity for solutions of \eqref{ueeq} is $C^{1,\alpha}$.

\medskip

In this paper, we study the singular set of a solution $u$ of \eqref{ueeq}, consisting of those points $x$ for which $u$ of \eqref{ueeq} fails to be $C^{2,\alpha}$ in any neighborhood of $x$. Our result asserts that the singular set has Hausdorff dimension at most $\d - \ep$, where the constant $\ep > 0$ depends only on the ellipticity of $F$ and  the dimension of the ambient space. The hypotheses (F1) and (F2) are stated in Section~\ref{P}.

\begin{thm} \label{main}
Assume that $F$ satisfies (F1) and (F2). Let $u \in C(\Omega)$ be a viscosity solution of \eqref{ueeq} in a domain $\Omega \subseteq \Rd$. Then there is a constant $\ep > 0$, depending only on $\d$, $\lambda$, and $\Lambda$, and a closed subset $\Sigma \subseteq \overline \Omega$ of Hausdorff dimension at most $\d - \ep$, such that $u \in C^{2,\alpha} (\Omega\setminus  \Sigma)$ for every $0 < \alpha < 1$.
\end{thm}

As far as we know, Theorem~\ref{main} is the first result which provides an estimate on the smallness of the singular set of a solution of a general uniformly elliptic, fully nonlinear equation. The constant $\ep > 0$ which appears in the statement of the theorem is the same $\ep$ as in the $W^{2,\ep}$ estimate of Lin~\cite{L}; see Remark~\ref{whatisep}.

\medskip

Let us describe the idea of the proof of Theorem \ref{main}. By differentiating \eqref{ueeq} and applying $W^{2,\ep}$ estimates to $Du$, we effectively obtain a $W^{3,\ep}$ estimate for the solution $u$. Precisely formulated, this implies that $u$ has a global second-order Taylor expansion almost everywhere and that the constant in front of the cubic error term lies in $L^\ep$. Near points possessing quadratic expansions, we apply a generalization of a result of Savin~\cite{S}, which asserts that any viscosity solution of \eqref{ueeq} that is sufficiently close to a quadratic polynomial must be $C^{2,\alpha}$. The $L^\ep$ integrability of the modulus of the quadratic expansion then restricts the Hausdorff dimension of the singular set.

\medskip

While $\ep$ does not depend on the modulus $\omega$ in (F2), the assumption that $F$ is $C^1$ with a uniformly continuous first derivative is crucial to our method of proof. In particular, Theorem~\ref{main} does not apply to Bellman-Isaacs equations, which have the form
\begin{equation} \label{BI}
F(D^2u) : = \inf_{\alpha\in \mathcal{I}} \sup_{\beta\in\mathcal{J}} \left( -\tr(A_{\alpha\beta} D^2u) \right) = 0.
\end{equation}
Such operators $F$ are positively homogeneous, and obviously any function which is positively homogeneous and differentiable at the origin is linear. Therefore, the assumption (F2) is incompatible with nonlinearity if $F$ is positively homogeneous. We do not know whether such a partial regularity result is true for equations of the form \eqref{BI}.

\medskip

In the next section, we state our notation and some preliminary results needed in the proof of Theorem~\ref{main}. In Sections~\ref{W2} and~\ref{FL}, we give complete arguments for the $W^{2,\ep}$ estimate and the $C^{2,\alpha}$ regularity for flat solutions of \eqref{ueeq}. The proof of Theorem~\ref{main} is presented in Section~\ref{PR}.

\section{Preliminaries} \label{P}

In this section we state our hypotheses and collect some standard ingredients needed in the proof of Theorem~\ref{main}.

\subsection*{Notation and hypotheses}
Let $\Mat$ denote the set of real $\d$-by-$\d$ matrices, and $\Sy\subseteq \Mat$ the set of symmetric matrices. Recall that the Pucci extremal operators are defined for constants $0 < \lambda \leq \Lambda$ and $M\in \Sy$ by
\begin{equation*}
\Pucci(M) : =  \sup_{\lambda \iden \leq A \leq \Lambda\iden} -\tr (AM) \quad \mbox{and} \quad \pucci(M) : =  \inf_{\lambda \iden \leq A \leq \Lambda\iden} -\tr (AM).
\end{equation*}
Throughout this paper, we assume the nonlinear operator $F:\Sy \to \R$ satisfies the following:
\begin{enumerate}
\item[(F1)] $F$ is uniformly elliptic and Lipschitz; precisely, we assume that there exist constants $0 < \lambda \leq \Lambda$ such that, for every $M,N\in \Sy$,
\begin{equation*}
\pucci(M-N) \leq F(M) - F(N) \leq \Pucci(M-N).
\end{equation*}
\item[(F2)] $F$ is $C^1$ and its derivative $DF$ is uniformly continuous, that is, there exists an increasing continuous function $\omega:[0,\infty) \to [0,\infty)$ such that $\omega(0) = 0$, and for every $M,N\in \Sy$,
\begin{equation*}
\left| DF(M) - DF(N) \right| \leq \omega\!\left( |M-N| \right).
\end{equation*}
\end{enumerate}
We call a constant \emph{universal} if it depends only on the dimension $\d$, the ellipticity constants $\lambda$ and $\Lambda$, and the modulus $\omega$. Throughout, $c$ and $C$ denote positive universal constants which may vary from line to line. We denote by $Q_{x,r}$ the cube centered at $x$ and of side length $2r$. That is, we define $Q_{x,r}:= \{ y \in \Rd: |y_i - x_i| \leq r \}$ and $Q_{r} := Q_{0,r}$. Balls are written $B(x,r): = \{ y\in \Rd: |y-x| < r\}$ and $B_r: = B(0,r)$.

\medskip

Recall that the \emph{Hausdorff dimension} of a set $E \subseteq \Rd$ is defined by
\begin{multline*}
\mathcal{H}_{\mathrm{dim}}(E) : = \inf \bigg\{ 0 \leq s < \infty : \mbox{for all} \ \delta > 0, \ \mbox{there exists a collection} \  \{ B(x_j,r_j)\}  \ \mbox{of balls}\\
\mbox{such that} \ \ E \subseteq \bigcup_{j=1}^\infty B(x_j,r_j) \quad \mbox{and} \quad \sum_{j=1}^\infty r_j^s < \delta \bigg\}.
\end{multline*}

\subsection*{Standard results}
In this subsection, we state three results needed below. Proofs of all three of these results can be found in \cite{CC}.  We first recall the statement of the Alexandroff-Bakelman-Pucci (ABP) inequality. We use the notation $u^+ = \max\{ 0 , u\}$ and $u^- : = - \min\{ 0 , u \}$.

\begin{prop}[ABP inequality] \label{ABP}
Assume that $B_R \subseteq \Rd$ and $f \in C(B_R) \cap L^\infty(B_R)$. Suppose that $u\in C(B_R)$ satisfies inequality
\begin{equation*}
\left\{ \begin{aligned}
& \Pucci(D^2u) \geq f & \mbox{in} & \ B_R, \\
& u \geq 0 & \mbox{on} & \ \partial B_R.
\end{aligned} \right.
\end{equation*}
Then 
\begin{equation*}
\sup_{B_R} u^- \leq C R \| f^+ \|_{L^\d(\{ \Gamma_u = u \})},
\end{equation*}
where $C$ is a universal constant and $\Gamma_u$ is the convex envelope in $B_{2R}$ of $-u^-$, where we have extended $u \equiv 0$ outside $B_R$. 
\end{prop}

We next recall an interior $C^{1,\alpha}$ regularity result for solutions of \eqref{ueeq}.

\begin{prop} \label{c1alph}
If $u$ is a viscosity solution of \eqref{ueeq} in $B_1$, then $u \in C^{1,\alpha}(\overline B_{1/2})$ for some universal $0 < \alpha <1$.
\end{prop}

Finally, we recall a consequence of the Calder\'on-Zygmund cube decomposition. This appeared in a slightly different form as \cite[Lemma 4.2]{CC}.

\begin{prop} \label{CZD}
Suppose that $D \subseteq E \subseteq Q_1$ are measurable and $0< \delta < 1$ is such that:
\begin{itemize}
\item $|D| \leq \delta |Q_1|$; and
\item if $x\in \Rd$ and $r> 0$ such that $Q_{x,3r} \subseteq Q_1$ and $|D\cap Q_{x,r}| \geq \delta |Q_{r}|$, then $Q_{x,3r} \subseteq E$.
\end{itemize}
Then $|D| \leq \delta |E|$.
\end{prop}

\section{$W^{2,\ep}$ estimate} \label{W2}
An integral estimate for the second derivatives of strong solutions of linear, uniformly elliptic equations in nondivergence form with only measurable coefficients was first obtained by Lin~\cite{L}. It was later extended to viscosity solutions of fully nonlinear equations in \cite{CC}.

\medskip

To state the estimate, we require some notation. Given a domain $\Omega \subseteq \Rd$ and a function $u\in C(\Omega)$, define the quantities
\begin{multline*}
\underline \Theta(x) = \underline \Theta(u,\Omega)(x): = \inf\big{\{} A \geq 0 : \ \mbox{there exists} \ p \in \Rd \ \mbox{such that for all} \ y \in \Omega, \\ 
u(y) \geq u(x) + p\cdot (x-y) - \tfrac12 A|x-y|^2 \big{\}},
\end{multline*}
\begin{multline*}
\overline \Theta(x) = \overline \Theta(u,\Omega)(x): = \inf\big{\{} A \geq 0 : \ \mbox{there exists} \ p \in \Rd \ \mbox{such that for all} \ y \in \Omega, \\ 
u(y) \leq u(x) + p\cdot (x-y) + \tfrac12 A|x-y|^2 \big{\}},
\end{multline*}
and
\begin{equation*}
\Theta(x) : = \Theta(u,\Omega)(x) = \max\left\{ \underline \Theta(u,\Omega)(x), \overline \Theta(u,\Omega)(x) \right\}.
\end{equation*}
The quantity $\underline \Theta(x)$ is the minimum curvature of any paraboloid that touches $u$ from below at $x$. If $u$ cannot be touched from below at $x$ by any paraboloid, then $\underline \Theta(x) = +\infty$. A similar statement holds for $\overline \Theta(x)$, where we touch from above instead.

\medskip

The form of the $W^{2,\ep}$ estimates we need is contained in the following proposition.

\begin{prop} \label{w2ep}
If $u \in C(B_1)$ satisfies the inequality
\begin{equation}
\Pucci(D^2u) \geq 0 \quad \mbox{in} \ B_1,
\end{equation}
then for all $t>t_0 \sup_{B_1} |u|$, 
\begin{equation} \label{w2epest}
\left| \left\{ x\in B_{1/2} : \underline\Theta(u,B_1)(x) > t\right\} \right| \leq Ct^{-\ep},
\end{equation}
%\begin{equation} \label{w2epest}
%\int_{B_{1/2}} \left( \underline\Theta(u,B_1)(x)\right)^\ep\, dx \leq C \sup_{B_1} |u|^\ep,
%\end{equation}
where the constants $C,t_0, \ep > 0$ are universal.
\end{prop}

Obviously \eqref{w2epest} implies that for any $0< \hat \ep < \ep$,
\begin{equation}
\int_{B_{1/2}} \left( \underline\Theta(u,B_1)(x)\right)^{\hat\ep}\, dx \leq C \sup_{B_1} |u|^{\hat\ep},
\end{equation}
where the constant $C$ depends additionally on a lower bound for $\ep - \hat \ep$. 

\medskip

Proposition~\ref{w2ep} is stated differently than the corresponding estimate in \cite{CC}. We emphasize here that $\underline\Theta(u,\Omega)(x)$ is defined in terms of quadratic polynomials which touch $u$ at $x$ and stay below $u$ in the full domain $\Omega$. That Proposition~\ref{w2ep} is stated in terms of such a quantity is crucial to its application in the proof of Theorem~\ref{main}. Indeed, if instead of Proposition~\ref{w2ep} we had the weaker statement that $u$ is twice differentiable at almost every point, and $|D^2 u| \in L^\ep$, then this would be insufficient to prove the partial regularity result. 

\medskip

For completeness and because of our alternative formulation, we give a simplified proof of Proposition~\ref{w2ep} following the along the lines of the argument in~\cite{CC}. The heart of the proof is following consequence of the ABP inequality. We recall that $\overline Q_1 \subset Q_3$ and $\overline Q_3\subset B_{6\sqrt\d}$.
\begin{lem} \label{w2heart}
Assume that $\Omega \subseteq \Rd$ is open and $\overline B_{6\sqrt\d} \subseteq \Omega$. Suppose $u\in C(\Omega)$ satisfies
\begin{equation}
\Pucci(D^2u) \geq 0 \quad \mbox{in} \ \Omega,
\end{equation}
such that for some $t> 0$,
\begin{equation*}
\{ \underline \Theta(u,\Omega) \leq t \} \cap Q_3 \neq \emptyset.
\end{equation*}
Then there are universal constants $M>1$ and $\sigma > 0$ such that 
\begin{equation} \label{w2hrtest}
\left| \{ \underline \Theta(u,\Omega) \leq M t \} \cap Q_1 \right| \geq \sigma > 0.
\end{equation}
\end{lem}
\begin{proof}
Since the operator $\Pucci$ and the quantity $\underline \Theta$ are positively homogeneous, we may assume that $t =1$. By adding an affine function to $u$, we may suppose that the paraboloid
\begin{equation*}
P(x): = \frac12 \left( 36 \d -  |x|^2 \right)
\end{equation*}
touches $u$ from below at some point $x_0 \in Q_3$, that is,
\begin{equation*}
\inf_{\Omega}  (u-P) = u(x_0) - P(x_0) = 0.
\end{equation*}
In particular,
\begin{equation*}
u \geq P \geq 0 \quad \mbox{in} \ B_{6 \sqrt\d} \qquad  \mbox{and} \qquad u(x_0) = P(x_0) \leq \sup_{Q_3} P = 18 \d.
\end{equation*}

According to \cite[Lemma 4.1]{CC}, there exist smooth functions  $\varphi$ and $\xi$ on $\Rd$ and universal constants $C$ and $K> 1$ such that
\begin{equation}
\left\{ \begin{aligned}
& \Pucci(D^2\varphi) \geq - C \xi \quad \mbox{in} \ \Rd, \\
& 0 \leq \xi \leq 1, \quad \xi \equiv 0 \ \  \mbox{on} \ \ \Rd \!\setminus \! Q_1,\\
& \varphi \geq -K \ \  \mbox{in} \ \Rd, \  \ \varphi \geq 0 \ \ \mbox{in} \ \Rd \!\setminus \! B_{6\sqrt{\d}},  \ \ \mbox{and} \ \ \varphi \leq -1 \ \ \mbox{in} \ Q_3.
&
\end{aligned}
\right.
\end{equation}
Define $w: = u + A\varphi$, with $A> 0$ selected below. It is easy to check that $w$ satisfies
\begin{equation*}
\left\{ \begin{aligned}
& \Pucci(D^2w) \geq -C A \xi & \mbox{in} & \ B_{6\sqrt\d}, \\
& w \geq 0 & \mbox{on} & \ \partial B_{6\sqrt\d}.
\end{aligned} \right.
\end{equation*}
Let $\Gamma_w$ denote the convex envelope of $- w^- \chi_{6 \sqrt \d}$ in $B_{12\sqrt\d}$. According to the ABP inequality (Proposition~\ref{ABP} above),
\begin{equation*}
-18 n + A \leq -u(x_0)-\varphi(x_0) \leq \sup_{B_{6\sqrt\d}} w^- \leq C A \left| \left\{ x\in Q_1 : \Gamma_w(x) = w(x) \right\} \right|.
\end{equation*}
Choosing $A = 19n$ yields
\begin{equation*}
\left| \left\{ x\in Q_1 : \Gamma_w(x) = w(x) \right\} \right| \geq \sigma, 
\end{equation*}
for a universal constant $\sigma> 0$. We also have that $u^- \leq C$ in $B_{6\sqrt\d}$. We finish the proof by showing that, for some universal constant $M > 0$, 
\begin{equation} \label{finhrtl}
\left\{ x\in Q_1 : \Gamma_w(x) = w(x) \right\} \subseteq \left\{ x \in Q_1 : \underline \Theta(u,\Omega)(x) \leq M \right\}.
\end{equation}
If $x\in Q_1$ is such that $\Gamma_w (x) = w(x)$, then since $\Gamma_w$ is convex and negative in $B_{12\sqrt\d}$, there exists an affine function $L$ that touches $\Gamma_w$, and hence $w$, from below at $x$. It follows that $L \leq 0$ in $B_{12\sqrt\d}$. We have 
\begin{equation*}
L -A\varphi \leq u \quad \mbox{in} \ B_{6\sqrt\d}
\end{equation*}
with equality holding at $x$. Since $L(x)= u(x) + A \varphi(x) \geq -KA$ and $L\geq 0$ in $B_{12 \sqrt \d}$, we deduce that $DL \leq KA/(6\sqrt\d)$. Since $|D^2\varphi|$ is bounded by a universal constant, we can find a concave paraboloid $\widetilde P$ with opening $M$, with $M$ universal, such that $\widetilde P \leq u$ in $B_{6\sqrt\d}$ and equality holding at $x$. Since $\dist(Q_1,\Rd \setminus B_{6\sqrt\d} ) \geq 5 \sqrt\d$ and $|DL|\leq C$, by making $M$ larger if necessary, we may assume that $\widetilde P \leq P$ on the set $\Rd \setminus B_{6\sqrt\d}$. Hence $\widetilde P \leq P \leq u$ on $\Omega$. Therefore $\widetilde P \leq u$ on $\Omega$ with equality holding at $x$. This completes the proof of \eqref{finhrtl}.
\end{proof}

We now prove Proposition~\ref{w2ep} by applying Proposition~\ref{CZD} to the contrapositive of Lemma~\ref{w2heart}.

\begin{proof}[Proof of Proposition~\ref{w2ep}]
According to the previous lemma, there are universal constants $M,\sigma > 0$, such that, for all $t> 0$ and $Q_{x,3r} \subseteq Q_{1/6\sqrt\d}$, we have
\begin{equation*}
\left| \{ \underline \Theta(u,B_1) > Mt \} \cap Q_{x,r} \right| > (1-\sigma)|Q_{r}| \qquad \mbox{implies that} \qquad \underline \Theta > t \quad \mbox{in} \ Q_{x,3r}.
\end{equation*}
Since $u$ is bounded, we can touch it from below in $Q_{1/2\sqrt\d}$ by a paraboloid with an opening proportional to $\sup_{B_1}|u|$, and hence for some $x \in Q_{1/2\sqrt\d}$,
\begin{equation*}
\underline \Theta(u,B_1)(x) \leq C \sup_{B_1} |u|.
\end{equation*}
Then according to Lemma~\ref{w2heart}, there exists a universal $t_0$ such that, for all $t> t_0 \sup_{B_1} |u|$,
\begin{equation*}
\left| \{ \underline \Theta(u,B_1) > Mt \} \cap Q_{1/6\sqrt\d} \right| \leq (1-\sigma) |Q_{1/6\sqrt\d}|.
\end{equation*}
It follows from Lemma~\ref{CZD} that, for every $t>t_0 \sup_{B_1} |u|$,
\begin{equation} \label{dec}
\left| \{ \underline \Theta(u,B_1) > Mt \} \cap Q_{1/6\sqrt\d} \right| \leq (1-\sigma) \left| \{ \underline \Theta(u,B_1) > t \} \cap Q_{1/6\sqrt\d} \right|.
\end{equation}
By iterating \eqref{dec}, we obtain a universal constants $C,\ep > 0$ such that for all $t>t_0\sup_{B_1} |u|$,
\begin{equation*}
\left| \{ \underline \Theta(u,B_1) > t \} \cap Q_{1/6\sqrt\d} \right| \leq Ct^{-\ep}.
\end{equation*}
%Therefore,
%\begin{align*}
%\int_{Q_{1/6\sqrt\d}} (\underline\Theta(u,B_1)(x))^\ep\, dx & = \int_0^\infty \ep t^{\ep-1}  \left| \{ \underline \Theta(u,B_1) > t \} \cap Q_{1/6\sqrt\d}\right|\, dt \\
%& \leq |Q_{1/6\sqrt\d}| \left( t_0 \sup_{B_1} |u|\right)^\ep + C\int_{ t_0\sup_{B_1}|u|}^\infty \ep t^{-1-\ep} \, dt \\
%& \leq C \sup_{B_1} |u|^\ep.
%\end{align*}
The proposition now follows from an easy covering argument.
\end{proof}

\begin{remark} \label{eprunaway}
It is natural to wonder what, if anything, can be said about the exponent $\ep$ in Proposition~\ref{w2ep}. By constructing an explicit example, we will show that $\ep \to 0$ as the ellipticity $\Lambda / \lambda \to \infty$.

\medskip

Fix $\alpha, R>0$, and define the function $u$ in $\R^2 \setminus \{ 0 \}$ by
\begin{equation*}
u(x) := \begin{cases}
R^{\alpha+2} |x|^{- \alpha} + \frac \alpha 2 |x|^2 - (1+\frac \alpha 2) R^2 & \mbox{if } 0 < |x| < R, \\
0 & \mbox{if } |x| \geq R.
\end{cases}
\end{equation*}
Observe that $u\in C^1(\R^2\setminus \{ 0 \})$. An easy computation confirms that for $0 < |x| < R$,
\begin{equation*}
D^2u(x) = \alpha(\alpha+2) R^{\alpha+2} |x|^{-\alpha-4} x \otimes x - \alpha|x|^{-\alpha-2}( R^{\alpha+2}-|x|^{\alpha+2} )\iden,
\end{equation*}
from which we can see that, in the punctured ball $0<|x|<R$, the eigenvalues of $D^2u(x)$ are $-\alpha|x|^{-\alpha-2}(R^{\alpha+2}-|x|^{\alpha+2})$ and $\alpha|x|^{-\alpha-2}(|x|^{\alpha+2}+ (\alpha+1)R^{\alpha+2})$. Therefore, 
\begin{multline} \label{epruneq}
\Pucci(D^2 u) = \Lambda \alpha|x|^{-\alpha-2}(R^{\alpha+2}-|x|^{\alpha+2}) - \lambda\alpha|x|^{-\alpha-2} \left( |x|^{\alpha+2} + (\alpha+1)R^{\alpha+2} \right) \\
\geq -(\Lambda+\lambda)\alpha  \geq -2\Lambda \alpha \quad \mbox{in} \ B_R \setminus \{ 0 \},
\end{multline}
provided that $0 < \alpha \leq \Lambda/\lambda -1$.

Since $u \equiv 0$ in $\R^2 \setminus B_R$, the inequality \eqref{epruneq} also holds in $\R^2\setminus B_R$. Using that $u \in C^1(\R^2 \setminus \{ 0 \})$, it follows that the inequality \eqref{epruneq} holds in the viscosity sense in $\R^2 \setminus \{0 \}$. 

For any neighborhood $N$ of $x \in B_R \setminus \{ 0 \}$, we have
\begin{equation} \label{bump-u}
\underline \Theta(u,N)(x) \geq \alpha|x|^{-\alpha-2} (R^{\alpha+2} - |x|^{\alpha+2}).
\end{equation}
This is easily deduced from the fact that if $\varphi$ is a smooth function touching $u$ from below at $x$, then $D^2\varphi(x) \leq D^2u(x)$, and the latter has an eigenvalue of $-\alpha|x|^{-\alpha-2} (R^{\alpha+2} - |x|^{\alpha+2})$. It follows that
\begin{equation*}
\underline \Theta(u,N)(x) \geq c \alpha R^{\alpha+2} |x|^{-\alpha-2} \quad \mbox{in} \ B_{R/2} \setminus \{ 0 \}
\end{equation*}
where $c$ depends only on $\alpha$. 

We build the example by making $R > 0$ small and replicating the function $u$.
\begin{equation*}
v(x) := - |x|^2 + \sum_{y \in \mathbb Z^2} \min\left(1, \frac {\lambda}{\Lambda \alpha} u(x-2Ry) \right)
\end{equation*}
Note that for $R$ small, the minimum inside the summation takes the second value when $|x-2Ry| > cR^{(\alpha+2)/\alpha}$.

It is routine to check that
\begin{equation*}
|v| \leq 1 \quad \mbox{in} \ B_1 \qquad \mbox{and} \qquad \Pucci(D^2 v) \geq 0 \quad \mbox{in}\  \R^2.
\end{equation*}
Now fix $\ep > 0$, and suppose that $\alpha > 0$ is large enough that $(\alpha+2)\ep > 2$. This can be arranged if the ellipticity satisfies $(\Lambda/\lambda+1)\ep > 2$. Using \eqref{bump-u}, it follows that for any $y \in \mathbb{Z}^2$ with $B(2Ry,R) \subseteq B_{1/2}$,
\begin{align*}
\int_{B(2Ry,R)} \big( \underline \Theta(v,B_1)(x)\big)^\ep \, dx &= \int_{B_R} \big( \underline \Theta(v,B_1)(x)\big)^\ep \, dx \\
& \geq \int_{B_{R/2} \setminus B_{cR^{(2+\alpha)/\alpha}}} \left( \underline \Theta(\tfrac1{\lambda \alpha}u,B_{1/2})(x-2Ry) \right)^{\ep}\, dx \\
&\geq \int_{B_{R/2} \setminus B_{cR^{(2+\alpha)/\alpha}}} \Big(c(\alpha,\lambda) R^{\alpha+2}|x|^{-\alpha-2}\Big)^\ep \, dx \\
& = 2\pi c(\alpha,\lambda)^\ep R^{(\alpha+2)\ep} \int_{cR^{(\alpha+2)/\alpha}}^{R/2} t^{-(\alpha+2)\ep +1} \, dt \\
&\geq c(\lambda,\alpha)^\ep R^{2(\alpha+2)(1-\ep)/\alpha} \qquad \text{for small } R.
\end{align*}
There exist $c/R^2$ disjoint balls of the form $B(2Ry,R)$, with $y \in \mathbb Z^2$, inside $B_{1/2}$. Therefore,
\[ \int_{B_{1/2}} \big( \underline \Theta(v,B_1)(x)\big)^\ep \, dx
\geq c(\alpha,\lambda)^\ep R^{2(\alpha+2)(1-\ep)/\alpha-2} = c(\alpha,\lambda)^\ep R^{2(2-(\alpha+2)\ep)/\alpha} . \]
Observe that the exponent $2(2-(\alpha+2)\ep)/\alpha < 0$. Thus $\| \underline \Theta(v,B_1) \|_{L^\ep(B_{1/2})}  \to +\infty$ as $R \to 0$, keeping $\lambda$, $\Lambda$, $\alpha$, and $\ep$ fixed.

This demonstrates that the $W^{2,\ep}$ estimate as stated in Proposition~\ref{w2ep} is false in dimension $n=2$ if we have $(\Lambda/\lambda+1)\ep > 2$. It is false in all dimensions $n\geq 2$, for the same range of $\ep$ and $\Lambda/\lambda$, since we may add dummy variables to our example at no cost. In particular, the exponent $\ep$ in Proposition~\ref{w2ep} is never greater than 1.
\end{remark}

\begin{conj} \label{ASSconj}
The optimal exponent in Proposition~\ref{w2ep} is $\ep = 2(\Lambda/\lambda+1)^{-1}$.
\end{conj}

It is not difficult to show that Conjecture~\ref{ASSconj} is true in the case that $\Lambda=\lambda$.

\section{$C^{2,\alpha}$ regularity for flat solutions} \label{FL}
We present a refinement of a result of Savin \cite{S}, which states that a viscosity solution of a uniformly elliptic equation that is sufficiently close to a quadratic polynomial is, in fact, a classical solution.

\begin{prop} \label{flatreg}
Assume in addition that $F(0) = 0$. Suppose that $0 < \alpha < 1$ and $u\in C(B_1)$ is a solution of \eqref{ueeq} in $B_1$. Then there exists a universal constant $\delta_0 = \delta_0(\alpha)> 0$, depending also on $\alpha$, such that 
\begin{equation*}
\| u \|_{L^\infty(B_1)} \leq \delta_0 \quad \mbox{implies that} \quad u \in C^{2,\alpha}(B_{1/2}).
\end{equation*}
Moreover, the following estimate holds
\[ ||u||_{C^{2,\alpha}(B_{1/2})} \leq C |u|_{L^\infty}. \]
\end{prop}

In the case that $F \in C^2$ and $|D^2F|$ is bounded, Proposition~\ref{flatreg} is a special case of Theorem~1.3 in \cite{S}. For completeness, and because we need the result under slightly different hypotheses on $F$, we give a proof of Proposition~\ref{flatreg} here, following the argument of~\cite{S}.

\medskip

The key step in the proof of Proposition \ref{flatreg} is given by the following Lemma.

\begin{lem} \label{savin-lemma2}
Suppose in addition that $F(0) = 0$ and fix $0 < \alpha < 1$. Then there exist universal constants $\delta_0 > 0$ and $0 < \eta < 1$, depending also on $\alpha$, such that, if $u\in C(B_1)$ is a solution of \eqref{ueeq} in $B_1$, and
\begin{equation*}
\sup_{B_1} |u| \leq \delta,
\end{equation*}
then there is a quadratic polynomial $P$ satisfying $F(D^2P) = 0$ and
\begin{equation} \label{qpd}
\sup_{B_\eta} |u-P| \leq \eta^{2+\alpha} \sup_{B_1} |u|.
\end{equation}
\end{lem}
\begin{proof}
We argue by compactness. With $\eta > 0$ to be chosen below, assume on the contrary that there exist sequences $\{F_k\}$ and $\{u_k\}$, such that:
\begin{itemize}
\item $F_k:\Sy\to \R$ satisfies (F1) and (F2) with the same $\lambda$, $\Lambda$, $\omega$, and $F_k(0) = 0$;
\item $u_k\in C(B_1)$ satisfies $F_k(D^2u_k) = 0$ in $B_1$;
\item $\delta_k:=\sup_{B_1} |u_k| \rightarrow 0$ as $k\to \infty$; and
\item there is no quadratic polynomial $P$ satisfying \eqref{qpd} for $u=u_k$. 
\end{itemize}
Using interior H\"older estimates and taking a subsequence, if necessary, we may suppose that there is an operator $F_0$ and a function $u_0 \in C(B_1)$ such that, as $k\to \infty$, we have the limits:
\begin{itemize}
\item $F_k \rightarrow F_0$ locally uniformly on $\Sy$;
\item $DF_k \rightarrow DF_0$ locally uniformly on $\Sy$; and
\item $m_k^{-1} u_k \rightarrow u_0$ locally uniformly in $B_1$.
\end{itemize}
We claim that $u_0$ is a solution of the constant coefficient linear equation
\begin{equation} \label{savlin}
DF_0(0) \cdot D^2u_0 = 0 \quad \mbox{in} \ B_1
\end{equation}
To verify \eqref{savlin}, select a smooth test function $\varphi$ and a point $x_0\in B_1$ such that
\begin{equation*}
x\mapsto (u_0 - \varphi)(x) \quad \mbox{has a strict local maximum at} \ x = x_0.
\end{equation*}
Then we can find a sequence $x_k \in B_1$ such that $x_k \rightarrow x_0$ as $k\to \infty$, and
\begin{equation*}
x \mapsto \left( u_k - \delta_k \varphi \right)(x) \quad \mbox{has a local maximum at} \ x = x_k.
\end{equation*}
Therefore,
\begin{equation*}
F_k\left(\delta_k D^2\varphi(x_k) \right) \leq 0.
\end{equation*}
Observe that 
\begin{multline*}
F_k\left(\delta_k D^2\varphi(x_k) \right) = \frac{d}{dt}\int_0^{\delta_k} F_k(tD^2\varphi(x_k)) \, dt  = \int_0^{\delta_k} DF_k(tD^2\varphi(x_k)) \cdot D^2\varphi(x_k) \, dt \\ \geq \delta_k DF_k(0) \cdot D^2\varphi(x_k) - \delta_k |D\varphi(x_k)|\, \omega\!\left( \delta_k |D^2\varphi(x_k)| \right).
\end{multline*}
Combining the last two inequalities, dividing by $\delta_k$, and letting $k\to \infty$ yields
\begin{equation*}
DF_0(0) \cdot D^2\varphi(x_0) \leq 0.
\end{equation*}
We have shown that $u_0$ is a subsolution of \eqref{savlin}, and checking that it is a supersolution is done by a similar argument.

Up to a change of coordinates, equation \eqref{savlin} is Laplace's equation. Since $\| u_0 \|_{L^\infty(B_1)} \leq 1$, standard estimates imply that $u_0 \in C^\infty(B_1)$, and that the quadratic polynomial $P(x): = u_0(0) + x \cdot Du_0(0) + x\cdot D^2u_0(0) x$ satisfies
\begin{equation} \label{unotqp}
\sup_{B_\eta} | u_0 - P | \leq C \eta^{3} \leq \frac{1}{2} \eta^{2+\alpha}.
\end{equation}
for a universal constant $\eta=\eta_\alpha> 0$, chosen sufficiently small and depending also on $\alpha$, and we also have
\begin{equation*}
DF_0(0) \cdot D^2P = 0.
\end{equation*}
Therefore,
\begin{equation*}
|DF_k(0) \cdot D^2P| = o(1) \quad \mbox{as} \ k \to \infty.
\end{equation*}
It follows that
\begin{multline*}
F_k(\delta_k D^2 P) = \frac{d}{dt} \int_0^{\delta_k} F_k(tD^2P) \, dt = \int_0^{\delta_k} DF_k(tD^2P) \cdot D^2P\, dt \\
\leq \delta_k DF_k(0) \cdot D^2P + \delta_k |D^2P| \, \omega \! \left( \delta_k|D^2P| \right) = o(\delta_k) \quad \mbox{as} \ k \to \infty.
\end{multline*}
Since $F_k$ is uniformly elliptic, we can find a constant $a_k\in \R$, of order $|a_k| = o(\delta_k)$, such that $P_k(x) : = \delta_k P(x) + a_k |x|^2$ satisfies $F(D^2P_k) = 0$. Using this, the uniform convergence of $u_k$ to $u_0$ on $B_\eta$, and multiplying \eqref{unotqp} by $\delta_k$, we see that for large enough $k$,
\begin{equation*}
\sup_{B_\eta} | u_k - P_k | \leq \delta_k \eta^{2+\alpha}. 
\end{equation*}
This contradiction completes the proof.
\end{proof}

\begin{proof}[Proof of Proposition~\ref{flatreg}]
By a standard covering argument, it is enough to show that the estimate holds at the origin. More precisely, we have to show that if $\|u\|_{L^\infty(B_1)} = \delta < \delta_0$, then there is a quadratic polynomial $P$ such that $F(D^2P)=0$, $|P| \leq C \delta$ in $B_1$ and 
\begin{equation} \label{P1}
|u(x) - P(x)| \leq C \delta |x|^{2+\alpha} \ \text{ for all } x \in B_1.
\end{equation}
The idea of the proof is to apply Lemma \ref{savin-lemma2}, in a decreasing sequence of scales, obtaining a sequence of quadratic polynomials approximating $u$ at zero with an appropriate error estimate.

Let $\eta \in (0,1)$ and $\delta_0 >0$ be as in Lemma \ref{savin-lemma2}. We will construct by induction a sequence of quadratic polynomials $\{P_k\}_{k=1}^\infty$ such that
\begin{equation} \label{Pk}
F(D^2 P_k)=0 \qquad \text{and} \qquad\| u-P_k\|_{L^\infty(B_{\eta^k})} \leq \delta \eta^{(2+\alpha)k}.
\end{equation}
Moreover, we will show that this sequence is convergent and its limit as $k \to \infty$ will be the desired polynomial $P$ giving the second order expansion of $u$ at the origin.

Since $||u||_{L^\infty(B_1)} = \delta$, $P_0=0$ suffices for the case $k=0$. Let us suppose that we have a quadratic polynomial $P_k$ for which \eqref{Pk} holds. Let $\tilde u$ and $\tilde F$ denote
\[ \tilde u(x) := \eta^{-2k} (u(\eta^kx) - P_k(\eta^kx)) \qquad \text{and} \qquad \tilde F(M) := F(M + D^2P_k).\]
Observe that $\tilde F(D^2 \tilde u)=0$ in $B_1$ and $|\tilde u| \leq \delta \eta^{k\alpha}$ in $B_1$. Applying Lemma \ref{savin-lemma2}, we find a quadratic polynomial $\tilde P_k$ such that $|\tilde P_k| \leq C \delta \eta^{k\alpha}$ in $B_1$ and
\[ \tilde F(D^2\tilde P_k)=0 \qquad \text{and} \qquad \|\tilde u - \tilde P_k\|_{B_\eta} \leq C \delta \eta^{k\alpha} \eta^{2+\alpha}=\delta \eta^{2+(k+1)\alpha} .\]
Let $P_{k+1} = P_k + \eta^{2k} \tilde P_k(\eta^{-k} x)$. From the estimate above, we have
\[ F(D^2 P_{k+1})=0 \qquad \text{and} \qquad \|u - P_{k+1}\|_{B_{\eta^{k+1}}} \leq C \delta \eta^{(k+1)(2+\alpha)} .\]
This completes the inductive construction of a sequence of polynomials satisfying \eqref{Pk}.

\medskip

It remains to show that the sequence $\{ P_k \}$ is convergent and that its limit $P$ satisfies \eqref{P1}. Since $|\tilde P_k| \leq C \delta \eta^{\alpha k}$ in $B_1$, its coefficients are bounded by $C \delta \eta^{\alpha k}$. More precisely, if $\tilde P = a_k + b_k \cdot x + x\cdot C_k x$, then $|a_k| + |b_k| + |C_k| \leq C \delta \eta^{\alpha k}$. Therefore
\[ P_{k+1} - P_k = \eta^{2k} \tilde P_k(\eta^{-k} x) = \eta^{2k} a_k + \eta^k b_k \cdot x + x^t C_k x. \]
Since $\eta<1$, all the coefficients of $P_{k+1}-P_k$ are bounded by the geometric series $C \delta \eta^{\alpha k}$. Therefore the sum $\sum_{k=1}^\infty (P_{k+1}-P_k)$ is telescoping and hence convergent, and we may define
\begin{equation*}
P := \lim_{k \to \infty} P_k = \sum_{k=1}^\infty (P_{k+1}-P_k).
\end{equation*}
Since $F(D^2 P_k)=0$ for every $k$ and $F$ is continuous, we also have $F(D^2 P)=0$.

\medskip

Writing $P(x) = a + b \cdot x + x^t C x$, we have the following estimates for the coefficients:
\begin{align*}
& |a-a_k| \leq \sum_{j=k}^\infty |a_{j+1}-a_j| \leq \sum_{j=k}^\infty C \eta^{(2+\alpha)j} \delta = C \eta^{(2+\alpha)k} \delta, \\
& |b-b_k| \leq \sum_{j=k}^\infty |b_{j+1}-b_j| \leq \sum_{j=k}^\infty C \eta^{(1+\alpha)j} \delta = C \eta^{(1+\alpha)k} \delta, \\
& |C-C_k| \leq \sum_{j=k}^\infty |C_{j+1}-C_j| \leq \sum_{j=k}^\infty C \eta^{\alpha j} \delta = C \eta^{\alpha k} \delta.
\end{align*}
Therefore $|P(x)-P_k(x)| \leq C \delta \eta^{(2+\alpha)k}$ if $x \in B_{\eta^k}$. In particular, $|P| = |P-P_0| \leq C \delta$ in $B_1$. Fix $x \in B_1$, and let $k$ be the integer so that $\eta^{k+1} < |x| \leq \eta^k$. Then we estimate
\begin{equation*}
|u(x) - P(x)| \leq |u(x) - P_k(x)| + |P_k(x)-P(x)| \leq C \delta \eta^{(2+\alpha)k} \leq C \delta |x|^{2+\alpha},
\end{equation*}
which completes the proof.
\end{proof}

\section{Partial regularity} \label{PR}

In this section, we prove our main result. We first apply the $W^{2,\ep}$ estimate in Proposition~\ref{w2ep} to the derivative of $u$, in effect deriving a $W^{3,\ep}$ estimate, and then to use this result and a scaling argument combined with Proposition~\ref{flatreg} to obtain the theorem.

\medskip

To state the $W^{3,\ep}$ estimate, we define the quantity
\begin{multline*}
\Psi(u,\Omega)(x)  :=  \inf\big{\{} A \geq 0 : \ \mbox{there exists} \ p \in \Rd \ \mbox{and} \ M\in \Mat \ \mbox{such that for all} \ y \in \Omega, \\ 
|u(y) - u(x) + p\cdot (x-y) +(x-y) \cdot M(x-y)| \leq \tfrac16 A|x-y|^3 \big{\}}.
\end{multline*}
The following lemma records an elementary relation between $\Psi(u,B_1)$ and $\Theta(u_{x_i},B_1)$.

\begin{lem}\label{theta-delta}
Assume that $u\in C^1(B_1)$. Then for each $x\in B_1$,
\begin{equation}
\Psi(u,B_1)(x) \leq \left( \sum_{i=1}^{\d}(\Theta(u_{x_i},B_1)(x))^2 \right)^{1/2}.
\end{equation}
\end{lem}
\begin{proof}
Suppose that $x\in B_1$ and $A_i \geq 0$ are such that $\Theta(u_{x_i},B_1)(x) \leq A_i$ for each $i=1,\ldots,\d$. Then we can find vectors $p^1,\ldots, p^n \in \Rd$ such that
\begin{equation} \label{thetie}
\left| u_{x_i}(y) - u_{x_i}(x) + p^i\cdot(x-y) \right| \leq \frac{1}{2} A_i|x-y|^2 \quad \mbox{for all} \ y\in B_1.
\end{equation}
Let $M\in \Mat$ be the matrix with entries $\frac12p^i_j$. It follows that 
\begin{multline} \label{inftoa}
u(y) - u(x) + Du(x)\cdot (x-y) + (x-y) \cdot M(x-y) \\
= (y-x) \cdot \int_0^1 Du(x+t(y-x)) - Du(x) + 2t M \cdot (y-x)\, dt.
\end{multline}
According to \eqref{thetie},
\begin{equation*}
\int_0^1 \left| u_{x_i} (x+t(y-x)) - u_{x_i} (x) + t p^i_j \cdot (y-x) \right|\,dt \leq \frac12 A_it^2|x-y|^2.
\end{equation*}
Denoting $A = (A_1,\ldots,A_n)$ and using the previous inequality and \eqref{inftoa}, we obtain
\begin{multline*}
\left| u(y) - u(x) + Du(x)\cdot (x-y) + (x-y) \cdot M(x-y) \right| \\ \leq (y-x) \cdot \int_0^1 \frac12 A t^2  |x-y|^2\, dt \leq \frac{1}{6} |A| |x-y|^3\,dt.
\end{multline*}
Thus $\Psi(u,B_1) \leq |A|$, as desired.
\end{proof}

In the next lemma, we formulate the $W^{3,\ep}$ estimate in an appropriate way for its application in the proof of Theorem~\ref{main}. A similar statement was used by Caffarelli and Souganidis~\cite{CSo} to obtain an algebraic rate of convergence for monotone finite difference approximations of uniformly elliptic equations.

\begin{lem} \label{mainlem1}
Suppose $u \in C(B_1)$ solves \eqref{ueeq} in $B_1$ and satisfies $\sup_{B_1} |u| \leq 1$. There are universal constants $C, \ep > 0$ such that, if $t > 1$, then
\begin{equation} \label{mainlem1est}
\left| \left\{ x\in B_{1/2} : \Psi(u,B_1)(x) > t \right\} \right| \leq C t^{-\ep}.
\end{equation}
\end{lem}
\begin{proof}
According to Proposition~\ref{c1alph}, we have that $u\in C^1(B_1)$. Moreover, according to \cite[Proposition 5.5]{CC}, for every unit direction $e\in \Rd$, $|e|=1$, the function $u_e:=e\cdot Du$ satisfies the inequalities
\begin{equation*}
\pucci(D^2 u_e) \leq 0 \leq \Pucci(D^2u_e) \quad \mbox{in} \ B_1,
\end{equation*}
in the viscosity sense. According to Proposition~\ref{w2ep}, we have, for each $t> 1$,
\begin{equation*}
\left| \left\{ x\in B_{1/2} : \underline\Theta(u,B_1)(x) > t\right\} \right| \leq Ct^{-\ep},
\end{equation*}
where $C,\ep > 0$ are universal constants. An application of Lemma~\ref{theta-delta} yields  \eqref{mainlem1est}.
\end{proof}

\begin{lem} \label{mainlem2}
Suppose that $u$ satisfies the hypotheses of Lemma~\ref{mainlem1} and that $0 < \alpha < 1$. There is a universal constant $\delta_\alpha > 0$, such that for every $y \in B_{1/2}$ and $0 < r < \tfrac1{16}$,
\begin{equation} \label{Delta-to-reg}
\left\{ \Psi(u,B_1) \leq r^{-1} \delta_\alpha \right\} \cap B(y,r) \neq \emptyset \qquad \mbox{implies that} \qquad u\in C^{2,\alpha}(B(y,r)).
\end{equation}
\end{lem}
\begin{proof}
Suppose that $0< r < \tfrac1{16}$, $y \in B_{1/2}$, and $z\in B(y,r)$ is such that
\begin{equation*}
\Psi (u,B_1)(z) \leq r^{-1}\delta. 
\end{equation*}
Then there exist $p \in \Rd$ and $M\in \Mat$ such that, for every $x\in B_1$,
\begin{equation} \label{touchz}
|u(x) - u(z) + p\cdot (z-x) +(z-x) \cdot M(z-x)| \leq \tfrac16 r^{-1} \delta |z-x|^3. 
\end{equation}
Replacing $M$ by $\frac12(M+M^t)$, we may assume that $M\in \Sy$. Since $u$ is a viscosity solution of \eqref{ueeq}, it is clear that $F(-M) = 0$. Define the function
\begin{equation*}
v(x): = \frac{1}{16r^2} \left( u(z+4rx) - u(z) + 4r p\cdot x + 16r^2 x\cdot Mx \right), \quad x\in B_{1}.
\end{equation*}
The inequality \eqref{touchz} implies that
\begin{equation*}
\sup_{B_1} |v| \leq \tfrac13 \delta.
\end{equation*}
Define the operator $\widetilde F(N) : = F(N-M)$, and observe $\widetilde F$ satisfies (F1) and (F2), with the same ellipticity constants $\lambda$, $\Lambda$, and modulus $\omega$, and $\widetilde F(0) = F(-M) = 0$. It is clear that $v$ is a solution of
\begin{equation*}
\widetilde F(D^2v) = 0 \quad \mbox{in} \ B_1.
\end{equation*}
Let $\delta_0> 0$ be the universal constant in Proposition~\ref{flatreg}, which also depends on $\alpha$. Suppose that $\delta \leq 3\delta_0$. Then Proposition~\ref{flatreg} yield that $v\in C^{2,\alpha}(B_{1/2})$, from which we deduce that $u \in C^{2,\alpha}(B(z,2r))$. Since $B(y,r) \subseteq B(z,2r)$, we are done.
\end{proof}

\begin{proof}[Proof of Theorem~\ref{main}]
By a standard covering argument, we may fix $0 < \alpha < 1$ and assume that $\Omega = B_1$, $u\in C(B_1)$ is bounded, and to show that $u\in C^{2,\alpha}(B_{1/2}\setminus \Sigma)$ for a set $\Sigma \subseteq \overline B_{1/2}$ with $\mathcal H_{\mathrm{dim}}(\Sigma) \leq \d-\ep$. Since, for every $t> 0$, the operator $F_t(M):= t^{-1} F(tM)$ satisfies both (F1) and (F2) with the same constants $\lambda$, $\Lambda$ but a different modulus $\omega$. Since the constant $\ep > 0$ we obtain does not depend on $\omega$, we may therefore assume without loss of generality that $\sup_{B_1} |u| \leq 1$.

\medskip

Let $\Sigma \subseteq B_{1/2}$ denote the set of points $x\in B_{1/2}$ for which $u\not\in C^{2,\alpha}(B(x,r))$, for every $r> 0$. Notice that $\Sigma$ is closed, and thus compact. Fix $0 < r < \tfrac{1}{16}$. According to the Vitali Covering Theorem, there exists a finite collection $\{ B(x_i,r) \}_{i=1}^m$ of disjoint balls of radius $r$, with centers $x_i \in \Sigma$, such that
\begin{equation*}
\Sigma \subseteq \bigcup_{i=1}^m B(x_i,3r).
\end{equation*}
Since $x_i\in \Sigma$, according to Lemma~\ref{mainlem2} there exists a universal constant $\delta$, also depending on $\alpha$, such that
\begin{equation*}
\Psi (u,B_1)(y) > r^{-1} \delta\quad \mbox{for every} \ y\in \bigcup_{i=1}^m B(x_i,r).
\end{equation*}
Applying Lemma~\ref{mainlem1}, we deduce that
\begin{equation*}
mr^{\d} \leq C m |B_{r}| \leq C r^{-\ep}
\end{equation*}
for universal constants $C,\ep> 0$. Therefore,
\begin{equation*}
\sum_{i=1}^m \left| B(x_i,3r) \right|^{\d-\ep} \leq C.
\end{equation*}
We deduce that $\mathcal{H}^{\d-\ep}(\Sigma) \leq C$ for a universal constant $C$. Therefore, $\mathcal{H}_{\mathrm{dim}}(\Sigma)\leq \d -\ep$.
\end{proof}

\begin{remark} \label{whatisep}
An inspection of the proof reveals that the codimension $\ep$ in Theorem~\ref{main} is equal to the exponent $\ep$ of the $W^{2,\ep}$ estimate of Proposition~\ref{w2ep}.  In particular, it does not depend on the modulus $\omega$ of $DF$. It follows that we could further reduce the dimension of the singular set if we could improve the exponent of the $W^{2, \ep}$ estimate. However, it is not possible to improve the exponent $\ep$ in the $W^{2,\ep}$ estimate, since as we saw in Remark~\ref{eprunaway}, the constant $\ep$ is at most $2(\Lambda/\lambda+1)^{-1}$.
\end{remark}

\subsection*{Acknowledgment}
The first author was partially supported by NSF Grant DMS-1004645, the second author by NSF grant DMS-1001629 and the Sloan Foundation, and the third author by NSF grant DMS-1004595.

\bibliographystyle{plain}
\bibliography{partialreg}

\end{document}